\documentclass[11pt]{amsart}

\usepackage{amsmath,amssymb,amsfonts,graphics}
\newcommand{\B}{\mathbb{B}}
\newcommand{\N}{\mathbb{N}}

\newcommand{\Zc}{\mathcal{Z}}
\newcommand{\Hc}{\mathcal{H}}
\newcommand{\R}{\mathbb{R}}
\newcommand{\C}{\mathbb{C}}
\newcommand{\T}{\mathbb{T}}

\newcommand{\dist}{\text{dist}}

\newcommand{\from}{\colon}

\newtheorem{thm}{Theorem}[section]

\newtheorem{lem}[thm]{Lemma}
\newtheorem{prop}[thm]{Proposition}
\theoremstyle{definition}
\newtheorem{defn}[thm]{Definition}
\theoremstyle{remark}

\newtheorem{rem}[thm]{Remark}
\numberwithin{equation}{section}

\title[Hyperreflexivity constants of the bounded $n$-cocycle spaces]
{Hyperreflexivity constants of the bounded $n$-cocycle spaces of group algebras and C$^*$-algebras}

\author[Ebrahim Samei]{Ebrahim Samei}

\author[Jafar Soltani Farsani]{Jafar Soltani Farsani}

\address{Department of Mathematics and Statistics,
University of Saskatchewan, 106 Wiggins Road, Saskatoon, SK, CANADA.}

\email{samei@math.usask.ca}

\email{jas637@mail.usask.ca}



\begin{document}

\maketitle

\begin{abstract}
We introduced the concept of strong property $(\B)$ with a constant for Banach algebras and, by applying certain analysis on the Fourier algebra of a unit circle, we show that all C$^*$-algebras and group algebras have the strong property $(\B)$ with a constant given by $288\pi(1+\sqrt{2})$. We then use this result to find a concrete upper bound for the hyperreflexivity constant of $\Zc^n(A,X)$, the space of bounded $n$-cocycles from $A$ into $X$, where $A$ is a C$^*$-algebra or the group algebra of a group with an open subgroup of polynomial growth and $X$ is a Banach $A$-bimodule for which $\Hc^{n+1}(A,X)$ is a Banach space. As another application, we show that for a locally compact amenable group $G$ and $1<p<\infty$, the space
$CV_P(G)$ of convolution operators on $L^p(G)$ are hyperreflexive with a constant given by $288\pi(1+\sqrt{2})$. This is the generalization of a well-known result of E. Christiensen in \cite{Christen} for $p=2$.
\end{abstract}

\section{Introduction}
The concept of hyperreflexivity is a strengthening of the well-known notion of reflexivity. The later notion that was first defined in \cite{AL28} has attracted much attentions during the years. It has its origin in the operator theory and at first was defined for the subspaces of $B(X)$. In \cite{E.S ref10}, D. R. Larson generalized the concept of reflexivity to the subspaces of $B(X,Y)$, where $X$ and $Y$ are Banach spaces. One goal was to study the local behavior of derivations from a Banach algebra $A$ to a Banach $A$-bimodule $X$.

Let $A$ be a Banach algebra and $X$ a Banach $A$-bimodule. One interesting question is under what conditions each local derivation from $A$ into $X$ is a derivation, or equivalently, when $Z^1(A,X)$ is algebraically reflexive. One could also study the continuous version of this question: when the space of bounded derivations from $A$ into $X$, namely $\Zc^1(A,X)$, is reflexive. B. E. Johnson showed that $Z^1(A,X)$ is algebraically reflexive when $A$ is a C$^*$-algebra (\cite{E.S ref8}). In \cite{E.S2}, the first named author generalized the concept of local derivations to the higher cohomologies and defined the local $n$-cocycles. He showed that if $A$ is a C$^*$-algebra and $X$ a Banach $A$-bimodule, then every bounded local $n$-cocycle from $A^{(n)}$ into $X$ is an $n$-cocycle. In a subsequent paper, he introduced the concept of reflexivity for bounded $n$-linear maps (\cite{E.S1}). He showed that if $G$ is a locally compact group with an open subgroup of polynomial growth and $X$ a Banach $L^1(G)$-bimodule, then $\Zc^n(L^1(G), X)$, the space of bounded $n$-cocycles from $A$ into $X$, is reflexive. More results related to these questions can be found in \cite{AL7, AL9, AL12, AL13, AL14, AL15, AL16, E.S ref9,  E.S ref10, AL21, E.S3}.

As it was pointed out above, the concept of hyperreflexivity is a strengthening of reflexivity. This concept was first introduced by Arveson in \cite{AL4} and proved to be powerful in operator theory.
For instance, mainly due to the work of E. Christensen, it is shown that injective von Neumann algebras
are hyperreflexive \cite{Christen}. It is not known whether one can remove ``injectivity" from the preceding statement. In fact, this is equivalent with several open problems in operator algebras including the Kadison's similarity problem \cite{Pis}. The first attempt in studying the hyperreflexivity for the space of
 derivations was done by V. Shulman in \cite{Sh} where he showed that $\Zc^1(A,A))$ is hyperreflexive for a C$^*$-algebra $A$ if $\Hc^2(A,A)=0$.
For group algebras, it was first shown in \cite{AEV2} that $\Zc^1(L^1(G),L^1(G))$ is hyperreflexive for each amenable SIN-group. In \cite{E.S1}, the first named author extended the preceding result and showed that $\Zc^1(L^1(G),X^*)$ is hyperreflexive if $G$ is an amenable locally compact group with an open subgroup which is of polynomial growth and $X$ is an essential Banach $L^1(G)$-bimodule. In particular, $\Zc^1(L^1(G),L^1(G))$ is hyperreflexive for such a group.
In \cite{AEV3}, the later result was extended further so that one could drop the assumption of ``amenability" 

Our goal in this article is to extend the concept of hyperreflexivity to the higher cocycles. For Banach spaces $X$ and $Y$, we first define hyperreflexivity for subspaces of $B^n(X,Y)$, taking into account the
$n$-linear form of these spaces. Then we focus on $\Zc^n(A,X)$, the space of bounded $n$-cocycles from a Banach algebra $A$ into a Banach $A$-bimodule $X$ and investigate when it can be hyperreflexive. Our investigation leads us to find sufficient conditions under which $\Zc^n(A,X)$
becomes hyperreflexive. We demonstrate that for a large classes of Banach algebras, including nuclear C$^*$-algebra and group algebras of groups with open subgroups of polynomial growth, these sufficient conditions hold which give evidence that our conditions are satisfactory. For the case when $n=1$, our results  include and, at the same time, generalize all the ones already obtained in the literature pointed out in the preceding paragraph.

We then use this result to show that if $A$ is a C$^*$-algebra or the group algebra of a group with an open subgroup of polynomial growth and if $X$ is a Banach $A$-bimodule for which $\Hc^{n+1}(A,X)$ is a Banach space, then the hyperreflexivity constant of $\Zc^n(A,X)$, the space of bounded $n$-cocycles from $A$ into $X$, is bounded by
$$C2^{n-1}(M^2 288\pi(1+\sqrt{2})+(M+1)^2)^{n+1}$$
where $M$ is bound for the local units of $A$ and $C$ is a constant satisfying
$$\dist(T, Z^n(A,X))\leq C\|\delta^n(T)\|, \ \ \ (T\in B^n(Z,X)).$$

We provide examples for which $M$ and $C$ are known. Consequently, we can find a concrete number as a bound for the hyperreflexivity constant of $\Zc^n(A,X)$.

\section{Preliminaries}\label{S:Prelimin}

Let $X$ and $Y$ be Banach spaces. For $n\in \N$, let $X^{(n)}$ be
the Cartesian product of $n$ copies of $X$, and let $L^n(X,Y)$ and
$B^n(X,Y)$ be the spaces of n-linear maps and bounded n-linear
maps from $X^{(n)}$ into $Y$, respectively.

Let $A$ be a Banach algebra, and let $X$ be a Banach $A$-bimodule.
An operator $D\in L(A,X)$ is a derivation if for all $a,b\in A$, $D(ab)=aD(b)+D(a)b$. For each
$x\in X$, the operator $ad_x\in B(A,X)$ defined by $ad_x(a)=ax-xa$
is a bounded derivation, called an inner derivation. Let $Z^1(A,X)$
and $\mathcal{Z}^1(A,X)$ be the linear spaces of derivations and
bounded derivations from $A$ into $X$, respectively. For $n\in \N$
and $T\in L^n(A,X)$, define
\begin{eqnarray*}
\delta^nT:(a_1,\ldots ,a_{n+1}) &\mapsto & a_1T(a_2,\ldots ,a_n)\\
&+ & \sum_{j=1}^n (-1)^jT(a_1,\ldots ,a_ja_{j+1},\ldots ,a_{n+1})\\
& + &  (-1)^{n+1}T(a_1,\ldots ,a_n)a_{n+1}.
\end{eqnarray*}
It is clear that $\delta^n$ is a linear map from $L^n(A,X)$ into
$L^{n+1}(A,X)$; these maps are the {\it connecting maps}. The
elements of $\ker \delta^n$ are the {\it n-cocycles}; we denote this
linear space by $Z^n(A,X)$. If we replace $L^n(A,X)$ with $B^n(A,X)$
in the above, we will have the `Banach' version of the connecting
maps; we denote them with the same notation $\delta^n$. In this
case, $\delta^n$ is a bounded linear map from $B^n(A,X)$ into
$B^{n+1}(A,X)$; these maps are the {\it bounded connecting maps}.
The elements of $\ker \delta^n$ are the {\it bounded n-cocycles}; we
denote this linear space by $\mathcal{Z}^n(A,X)$. It is easy to
check that $Z^1(A,X)$ and $\mathcal{Z}^1(A,X)$ coincide with our
previous definition of these notations.

Let $A$ be a Banach algebra, and let $X$ be a Banach $A$-bimodule.
By \cite[Section 2.8]{D}, for $n\in \N$, the Banach space $B^n(A,X)$
turns into a Banach $A$-bimodule by the actions defined by:
\begin{eqnarray*}
(a\star T)(a_1,\ldots , a_n) &= & aT(a_1,\ldots , a_n);
\end{eqnarray*}
\begin{eqnarray*}
(T\star a)(a_1,\ldots ,a_n) &= & T(aa_1,\ldots ,a_n)  \\
&+ & \sum_{j=1}^n (-1)^jT(a,a_1,\ldots ,a_ja_{j+1},\ldots ,a_n)\\
& + & (-1)^{n+1}T(a,a_1,\ldots ,a_{n-1})a_n.
\end{eqnarray*}
In particular, when $n=1$, $B(A,X)$ becomes a Banach $A$-bimodule
with respect to the products
$$(a\star T)(b)=aT(b) \ \ , \ (T\star a)(b)=T(ab)-T(a)b.$$
Let $\Lambda_n \from B^{n+1}(A,X) \to B^n(A,B(A,X))$ be the
identification given by
$$(\Lambda_n(T)(a_1,\ldots ,a_n))(a_{n+1})=T(a_1,\ldots
,a_{n+1}).$$ Then $\Lambda_n$ is an $A$-bimodule isometric
isomorphism. If we denote the connecting maps for the complex
$B^n(A,(B(A,X),\star))$ by $\Delta^n$, then it is shown in \cite{D} that
\begin{align}\label{Eq:connecting map-higher to lower}
\Lambda_{n+1}\circ \delta^{n+1}=\Delta^n\circ \Lambda_n.
\end{align}


\section{A constant for the strong property $(\B)$}

The concept of the strong property $(\B)$ first appeared in \cite{AEV} for C$^*$-algebras and group algebras, where it was shown that they all possess this property. However, it was formally formulated and introduced in \cite{ES-JS} for general Banach algebras and was used to obtain hyperreflexivity of
bounded $n$-cocyle spaces from various Banach algebras. Since we are looking for further information such as a bound for the hyprereflexivity constant, we require a more refined version of the strong property $(\B)$, i.e. when its associated function is a line as described below.

\begin{defn}\label{def to B strong with constant}
We say that a Banach algebra $A$ has {\it the strong property $(\B)$ with a constant} $r>0$ if for each Banach space $X$ and every bounded bilinear map $\varphi:A\times A\rightarrow X$ with the property that
$$a,b\in A \ \ ab=0 \Rightarrow \|\varphi(a,b)\|\leq \alpha \|a\| \|b\|,$$
We can infer that
$$\|\varphi(ab,c)-\varphi(a,bc)\|\leq r \alpha \|a\|\|b\|\|c\| \ \ \ (\forall a,b,c\in A).$$
In other word,
$$\|\varphi(ab,c)-\varphi(a,bc)\|\leq r \alpha(\varphi) \|a\|\|b\|\|c\|, \ \ \ (\forall a,b,c\in A)$$
where
$$\alpha(\varphi)=\sup \{\|\varphi(a,b)\|: \ \ a,b\in A, \ \|a\|, \|b\|\leq 1, \ ab=0\}.$$
We note that by a simple application of Hahn-Banach theorem, it suffices to check the preceding property for the case when $X=\C$. We will use this alternative definition when it is more convenient.
\end{defn}

We will see later in Section \ref{hyper constant c-star and grp} that existence of a constant for the strong property $(\B)$ is fundamental in finding an upper bound for the hypereflexivity constant of the bounded $n$-cocycle spaces.

\subsection{Fourier algebra of the unit circle}\label{constant for the strong property (B)circle}

As it was mentioned above, in order to achieve our goal in finding an upper bound for the hyperreflexivity constant of the bounded $n$-cocycle spaces of C$^*$-algebras and group algebras, we need to find a constant for the strong property $(\B)$ for such Banach algebras. In the present section we aim to find such a constant for the Fourier algebra of the unit circle. Interestingly, we only need to study this case to find a constant for the strong property $(\B)$ of C$^*$-algebras and group algebras (see Theorem \ref{C* alg and group alg constant B}).

Let $\mathbb{T}$ denote the unit circle in $\mathbb{C}$, i.e.
$$\mathbb{T}=\{z\in \mathbb{C}: \ |z|=1\}.$$
Here we identify $\mathbb{T}$ with ${\mathbb{R}}/{\mathbb{Z}}\cong [-\pi,\pi]$. In this case $s=t$ if $s\equiv t (mod\,2\pi\mathbb{Z}).$
For every $f\in L^1(\mathbb{T})$, the Fourier transform on $f,$ denoted by $\hat{f},$ is defined by
$$\hat{f}(n)=\frac{1}{2\pi}\int_{-\pi}^{\pi}f(t)e^{-int}dt, \ \ \ (n\in\mathbb{Z}).$$
The Fourier algebra of the unit circle is defined as follows
$$A(\mathbb{T})=\{f\in L^1(\mathbb{T}): \ \|f\|_{A(\mathbb{T})}=\sum_{n\in\mathbb{Z}}|\hat{f}(n)|<\infty\}.$$
It is well-known that $A(\mathbb{T})\subseteq C(\mathbb{T})$, the space of continuous functions on $\mathbb{T}.$ Also $A(\mathbb{T})$ with the pointwise addition and multiplication and the norm $\|\cdot\|_{A(\mathbb{T})}$ is a Banach algebra.

The following lemma is essential for us to get our result.

\begin{lem}\label{lem T hyp const}
Let $X$ be a Banach space and $F:A(\T)\rightarrow X$ a linear map with $\|F\|\leq 1$. Suppose that $0\leq\alpha\leq1$ is such that for each $\varphi, \psi\in A(\T)$ with $\text{supp}\,\varphi\cap \text{supp}\,\psi=\emptyset$, we have
$$ \|F(\varphi*\check{\psi})\|\leq \alpha \|\varphi\| \|\psi\|.$$
Let $f\in A(\T)$ be given by $f(s)=e^{is}-1.$ Then
$$\|F(f)\|\leq 12\sqrt{\pi(1+\sqrt{2})}\sqrt{\alpha}.$$
\end{lem}
\begin{proof}
Let $0<\epsilon<3$. Define
$$W_{\epsilon}=\{x\in \T: \ \|f-R_xf\|_{A(\T)}< \epsilon\},$$
where $(R_xf)(s)=f(s+x).$
Note that for $s\in \T$
$$(f-R_xf)(s)=e^{is}(1-e^{ix}).$$
Hence if we define $e_1(s)=e^{is},$ then
$$\|f-R_xf\|_{A(\T)}=\|e_1\| |1-e^{ix}|=|1-e^{ix}|.$$
So
$$W_{\epsilon}=\{x\in \T: |1-e^{ix}|<\epsilon\}.$$
We show that for each $0<\delta<\epsilon$, $[-(\epsilon-\delta),(\epsilon-\delta)]\subseteq W_{\epsilon}.$ Define $g:[-\pi,\pi]\rightarrow \T$ by $g(s)=1-e^{is}.$
Let $0<x<\pi$. Applying vector-valued mean value theorem to the function $g|_{[0,x]}$, we find $0<c<x$ with
$$|g(x)|=|g(x)-g(0)|\leq|g^{\prime}(c)| |x|\leq |x|.$$
If $-\pi<x<0$, we use the same argument on the interval $[x,0]$. For $x=0$, the inequality trivially holds. So for each $0<\delta<\epsilon$ and for all
$x\in [-(\epsilon-\delta), (\epsilon-\delta)]$ we get
$$|e^{ix}-1|=|g(x)|\leq |x|< \epsilon.$$
It means that $[-(\epsilon-\delta),(\epsilon-\delta)]\subseteq W_{\epsilon}.$ Define
$$V_{\epsilon, \delta}=[\frac{-(\epsilon-\delta)}{3},\frac{(\epsilon-\delta)}{3}], \ \ U_{\epsilon, \delta}=[\frac{-(\epsilon-\delta)}{6},\frac{(\epsilon-\delta)}{6}].$$
Then $V_{\epsilon, \delta}+V_{\epsilon, \delta}+V_{\epsilon, \delta}\subseteq W_{\epsilon}$ and $U_{\epsilon, \delta}+U_{\epsilon, \delta}= V_{\epsilon,\delta}$. Now put
$$u=\frac{1}{\lambda(U_{\epsilon,\delta})^2}1_{U_{\epsilon,\delta}}*1_{U_{\epsilon,\delta}}$$
and
\begin{align}\label{definition v}
v=f(\frac{1}{\lambda(V_{\epsilon,\delta})}1_{V_{\epsilon,\delta}+V_{\epsilon,\delta}}*1_{V_{\epsilon,\delta}}).
\end{align}
Obviously, $1_{U_{\epsilon,\delta}}\in L^2(\T).$ Since $A(\T)=L^2(\T)*L^2(\T)$, we have $u\in A(\T)\subseteq C(\T)\subseteq L^2(\T)\subseteq  L^1(\T).$
It is easy to check that $\|1_{U_{\epsilon,\delta}}\|_2=\sqrt{\lambda(U_{\epsilon,\delta})}.$
By definition of the Fourier norm,
\begin{equation}\label{fourier norm of u}
\begin{aligned}
\|u\|_{A(\T)}&\leq& \frac{1}{\lambda(U_{\epsilon,\delta})^2} \|1_{U_{\epsilon,\delta}}\|_2 \|1_{U_{\epsilon,\delta}}\|_2\\
&=&\frac{1}{\lambda(U_{\epsilon,\delta})}=\frac{6\pi}{\epsilon-\delta}.
\end{aligned}
\end{equation}
Since $1_{U_{\epsilon,\delta}}\in L^2(\T)\subseteq L^1(\T)$ and $L^2(\T)$ is $L^1(\T)$-module with respect to the convolution,
\begin{align}\label{2-norm of u--1}
\|u\|_2\leq \frac{1}{\lambda(U_{\epsilon,\delta})^2} \|1_{U_{\epsilon,\delta}}\|_1 \|1_{U_{\epsilon,\delta}}\|_2.
\end{align}
It is easy to check that $\|1_{U_{\epsilon,\delta}}\|_1=\lambda(U_{\epsilon,\delta}).$ So by \eqref{2-norm of u--1},
\begin{align}\label{2-norm of u}
\|u\|_2\leq \frac{\lambda(U_{\epsilon,\delta})^{\frac{3}{2}}}{\lambda(U_{\epsilon,\delta})^2}=\frac{1}{\lambda(U_{\epsilon,\delta})^{\frac{1}{2}}}=\sqrt{\frac{6\pi}{\epsilon-\delta}}.
\end{align}
We show that $\text{supp}\,u\subseteq U_{\epsilon,\delta}+U_{\epsilon,\delta}.$ Let $x\in[-\pi,\pi]$. Then
\begin{eqnarray*}
(1_{U_{\epsilon,\delta}}*1_{U_{\epsilon,\delta}})(x)&=&
\frac{1}{2\pi}\int_{-\pi}^{\pi}1_{U_{\epsilon,\delta}}(y)1_{U_{\epsilon,\delta}}(x-y)dy\\
&=&\frac{1}{2\pi}\int_{U_{\epsilon,\delta}}1_{U_{\epsilon,\delta}}(x-y)dy.
\end{eqnarray*}
So for $x$ to be in $\text{supp}\,u$, there should exist $y\in \text{supp}\,1_{U_{\epsilon,\delta}}=U_{\epsilon,\delta}$ such that $x-y \in \text{supp}\ 1_{U_{\epsilon,\delta}}=U_{\epsilon,\delta}$. So
$x\in U_{\epsilon,\delta}+U_{\epsilon,\delta}$.\\
We also have
\begin{align}\label{1-norm of u}
\|u\|_1=\frac{1}{2\pi}\int_{-\pi}^{\pi}u(x) dx=\frac{1}{\lambda(U_{\epsilon,\delta})^2} \|1_{U_{\epsilon,\delta}}\|_1 \|1_{U_{\epsilon,\delta}}\|_1=1.
\end{align}
Next we prove some properties related to $v$ defined in \eqref{definition v}.\\
First of all, note that $1_{V_{\epsilon,\delta}+V_{\epsilon,\delta}}, 1_{V_{\epsilon,\delta}}\in L^2(\T).$ So $1_{V_{\epsilon,\delta}+V_{\epsilon,\delta}}* 1_{V_{\epsilon,\delta}}\in A(\T)$ which implies that $v\in A(\T).$ Also
\begin{eqnarray*}
\|v\|_{A(\T)}&\leq& \|f\|_{A(\T)} \|\frac{1}{\lambda(V_{{\epsilon,\delta}})}1_{V_{\epsilon,\delta}+V_{\epsilon,\delta}}*1_{V_{\epsilon,\delta}} \|_{A(\T)}\\
&\leq& \frac{1}{\lambda(V_{{\epsilon,\delta}})}\|f\|_{A(\T)} \|1_{V_{\epsilon,\delta}+V_{\epsilon,\delta}}\|_2 \|1_{V_{\epsilon,\delta}}\|_2.
\end{eqnarray*}
Obviously, $\|1_{V_{\epsilon,\delta}+V_{\epsilon,\delta}}\|_2=\sqrt{\lambda(V_{\epsilon,\delta}+V_{\epsilon,\delta})}$ and $\|1_{V_{\epsilon,\delta}}\|_2=\sqrt{\lambda(V_{\epsilon,\delta})}$. So
\begin{equation}\label{fourier norm of v}
\begin{aligned}
\|v\|_{A(\T)}&\leq& \|f\|_{A(\T)} (\frac{\lambda(V_{\epsilon,\delta}+V_{\epsilon,\delta})}{\lambda(V_{\epsilon,\delta})})^{\frac{1}{2}}\\
&=& 2(\frac{\frac{4(\epsilon-\delta)}{6\pi}}{\frac{2(\epsilon-\delta)}{6\pi}})^{\frac{1}{2}}=2\sqrt{2}.
\end{aligned}
\end{equation}
Using \eqref{fourier norm of v}, we can write
\begin{equation}\label{fourier norm of f-v}
\begin{aligned}
\|f-v\|_{A(\T)}&\leq&\|f\|_{A(\T)}+\|v\|_{A(\T)}\\
&\leq& 2(1+\sqrt{2}).
\end{aligned}
\end{equation}
Similar to what we proved for $u$, we have
$$\text{supp}\, v\subseteq \text{supp}\,(1_{V_{\epsilon,\delta}+V_{\epsilon,\delta}}*1_{V_{\epsilon,\delta}})\subseteq V_{\epsilon,\delta}+V_{\epsilon,\delta}+V_{\epsilon,\delta}\subseteq W_{\epsilon}.$$
We now show that for each $x\in V_{\epsilon,\delta}$, $f(x)=v(x).$ To see this, take $x\in V_{\epsilon,\delta}$. Then
\begin{eqnarray*}
(1_{V_{\epsilon,\delta}+V_{\epsilon,\delta}}*1_{V_{\epsilon,\delta}})(x)&=&\frac{1}{2\pi}\int_{-\pi}^{\pi}1_{V_{\epsilon,\delta}+V_{\epsilon,\delta}}(x-w)1_{V_{\epsilon,\delta}}(w)dw\\
&=&\frac{1}{2\pi}\int_{V_{\epsilon,\delta}}1_{V_{\epsilon,\delta}+V_{\epsilon,\delta}}(x-w)dw\\
&=&\lambda(V_{\epsilon,\delta})
\end{eqnarray*}
Hence $f(x)=v(x).$ This implies that
\begin{align}\label{333333}
\text{supp}\, (f-v)\subseteq V_{\epsilon,\delta}^c.
\end{align}
We show that $\|v\|_2\leq2\epsilon\sqrt{\frac{\epsilon-\delta}{3}}.$ Let $x\in W_{\epsilon}$. Then
\begin{eqnarray*}
|f(x)|&=&|f(0)-R_xf(0)|\\
&\leq&\|f-R_xf\|_{\infty}\\
&\leq&\|f-R_xf\|_{A(\T)}\\
&<&\epsilon.
\end{eqnarray*}
Since $\text{supp}\,v\subseteq W_{\epsilon}$, we get
\begin{eqnarray*}
\|v\|_{2}^{2}&=& \frac{1}{2\pi}\int_{W_{\epsilon}}|f(t)|^2 |\frac{1}{\lambda(V_{{\epsilon}})}1_{V_{\epsilon,\delta}+V_{\epsilon,\delta}}*1_{V_{\epsilon,\delta}}(t)|^2 dt\\
&\leq&\epsilon^2 \frac{1}{\lambda(V_{\epsilon,\delta})^2}\|1_{V_{\epsilon,\delta}+V_{\epsilon,\delta}}*1_{V_{\epsilon,\delta}}\|_2^2\\
&\leq&\epsilon^2 \frac{1}{\lambda(V_{{\epsilon,\delta}})^2}\|1_{V_{\epsilon,\delta}+V_{\epsilon,\delta}}\|_{2}^{2} \|1_{V_{\epsilon,\delta}}\|_{1}^{2}\\
&=&\epsilon^2 \frac{1}{\lambda(V_{{\epsilon,\delta}})^2}\lambda(V_{{\epsilon}}+V_{{\epsilon}})\lambda(V_{{\epsilon,\delta}})^2\\
&=&\epsilon^2\frac{4(\epsilon-\delta)}{6\pi}.
\end{eqnarray*}
This implies that
\begin{align}\label{2-norm of v}
\|v\|_2\leq 2\epsilon \sqrt{\frac{\epsilon-\delta}{6\pi}}.
\end{align}
We now show that $\|f-f*\check{u}\|_{A(\T)}\leq  \epsilon$. We can write $f*\check{u}$ as a Bochner integral
$$f*\check{u}=\frac{1}{2\pi}\int_{-\pi}^{\pi}{u}(x)R_xfdx.$$
By \eqref{1-norm of u}, $\frac{1}{2\pi}\int_{-\pi}^{\pi} u(x)dx=1.$ Therefore
\begin{eqnarray*}
\|f-f*\check{u}\|_{A(\T)}&=&\frac{1}{2\pi}\|\int_{-\pi}^{\pi}(f-R_x f){u}(x)dx\|_{A(\T)}\\
&\leq&\frac{1}{2\pi}\int_{U_{\epsilon,\delta}+U_{\epsilon,\delta}}\|(f-R_x f)\|_{A(\T)} |{u}(x)| dx\\
&<& \epsilon,
\end{eqnarray*}
where the last inequality follows from the fact that ${U_{\epsilon,\delta}+U_{\epsilon,\delta}}\subseteq W_{\epsilon}$ and $\|u\|_1=1.$
On the other hand, using \eqref{2-norm of u} and \eqref{2-norm of v}, we get
\begin{eqnarray*}
\|v*\check{u}\|_{A(\T)}&\leq&\|u\|_2 \|v\|_2\\
&\leq& \sqrt{\frac{6\pi}{\epsilon-\delta}}2\epsilon \sqrt{\frac{\epsilon-\delta}{6\pi}}=2\epsilon.
\end{eqnarray*}
So if we put $a=(f-v)*\check{u}$, then
\begin{equation}\label{fourier norm of f-a}
\begin{aligned}
\|f-a\|_{A(\T)}&\leq& \|f-f*\check{u}\|_{A(\T)}+\|v*\check{u}\|_{A(\T)}\\
&<& \epsilon+2\epsilon= 3\epsilon.
\end{aligned}
\end{equation}
Now we can write
\begin{eqnarray*}
\|F(f)\|&=&\|F(f-a+a)\|\\
&\leq& \|F(f-a)\|+\|F(a)\|.
\end{eqnarray*}
Since $a=(f-v)*\check{u}$ and by \eqref{333333}, $\text{supp}\,(f-v)\cap \text{supp}\,\check{u}\subseteq V_{\epsilon,\delta}^c\cap V_{\epsilon,\delta}=\emptyset$, we have (by hypothesis)
$$\|F(a)\|\leq \alpha\|f-v\|_{A(\T)}\|u\|_{A(\T)}.$$
Hence
$$\|F(f)\|\leq \|f-a\|_{A(\T)}+\alpha\|f-v\|_{A(\T)}\|u\|_{A(\T)}.$$
Using \eqref{fourier norm of u}, \eqref{fourier norm of f-v} and \eqref{fourier norm of f-a}, we get
$$\|F(f)\|\leq 3\epsilon+\alpha 2(1+\sqrt{2})\frac{6\pi}{\epsilon-\delta}\ \ \ (0<\epsilon<3 \ \ \ 0<\delta<\epsilon) .$$
 Letting $\delta\rightarrow 0$, $A=3$ and $B=12\pi(1+\sqrt{2})$, we have
 \begin{align}\label{main}
\|F(f)\|\leq \inf \{A \epsilon+\frac{\alpha B}{\epsilon}, \ \ \ 0<\epsilon< 3 \}.
\end{align}
 Define $k:(0,3)\rightarrow R^{+}$ by $k(\epsilon)=A \epsilon+\frac{\alpha B}{\epsilon}$. Then
 $$k^{\prime}(\epsilon)=A-\frac{\alpha B}{\epsilon^2}=0\Rightarrow \epsilon=\sqrt{\frac{\alpha B}{A}}.$$
 Note that for each $0\leq\alpha\leq1$ we have $\sqrt{\frac{\alpha B}{A}}<3$ . So by \eqref{main} we can write
 $$\|F(f)\|\leq k(\sqrt{\frac{\alpha B}{A}}) =2\sqrt{AB\alpha}= 12\sqrt{\pi(1+\sqrt{2})}\sqrt{\alpha}.$$
\end{proof}

We are now ready to prove the first main result of this section which, in part, implies that $A(\T)$ has the strong property $(\B)$ with the constant $288\pi(1+\sqrt{2})$. The method of our proof was partly inspired by \cite[Lemma 3.1]{AEV} and its proof but it goes further to provides a concrete constant for the strong
property $(\B)$.

\begin{thm}\label{mainn thm for strong (B) constant of A(T)}
Let $\phi:A(\T)\times A(\T)\rightarrow \mathbb{C}$ be a continuous bilinear map satisfying the property
\begin{align}\label{prop}
f,g\in A(\T), \ \text{supp}\,f\cap \text{supp}\,g=\emptyset\Rightarrow |\phi(f,g)|\leq \alpha \|f\|\|g\|
\end{align}
for some $\alpha\geq0.$ Then
\begin{align}\label{result}
|\phi(fg,h)-\phi(f,gh)|\leq 288\pi(1+\sqrt{2}){\alpha} \|f\|\|g\|\|h\|
\end{align}
for all $f,g,h\in A(\T).$
\end{thm}
\begin{proof}
First assume that $0\leq\alpha< 1$ and $\|\phi\|\leq 1$. The map $\phi$ gives rise to a continuous linear operator $\Phi$ on the projective tensor product $A(\T)\hat{\otimes} A(\T)(=A(\T\times\T))$ defined through
\begin{align}\label{gives rise}
\Phi(f\otimes g)=\phi(f,g) \ \ \ (f,g\in A(\T)).
\end{align}
We define $N: A(\T)\rightarrow A(\T\times\T)$ with
$$Nk(s,t)=k(s-t) \ \ (k\in A(\T), \ s,t\in\T).$$
Pick $f,h\in A(\T)$ with $\|f\|, \|h\|\leq 1$ and define $N_{f,h}: A(\T)\rightarrow A(\T\times\T)$ with
$$N_{f,h}k=Nk(f\otimes e_1h)$$
where $e_1\in A(\T)$ is given by $e_1(s)=e^{is}.$  Then it is easy to check that
\begin{align}\label{Nf}
N_{f,h}(e_1-1)=fe_1\otimes h-f\otimes e_1 h.
\end{align}
Note that for $\psi,\varphi\in A(\T),$ we have the Bochner integral equality
$$N(\varphi*\check{\psi})=\int_{\T}R_x\varphi\otimes R_x{\psi} dx.$$
Hence
\begin{align}\label{Bochner def}
N_{f,h}(\varphi*\check{\psi})=\int_{\T}(R_x\varphi) f\otimes (R_x{\psi})e_1h dx.
\end{align}
If $\text{supp}\,\varphi\cap \text{supp}\,\psi=\emptyset,$ then we have
$$\text{supp}\,((R_x\varphi) f)\cap \text{supp}\,((R_x{\psi})e_1 h)=\emptyset.$$
Hence using \eqref{Bochner def} we get
\begin{eqnarray*}
|\Phi \circ N_{f,h}(\varphi*\check{\psi})|&\leq& \int_{\T}\|\Phi((R_x\varphi) f\otimes (R_x{\psi})e_1 h)\|dx\\
&\leq&\int_{\T}\|\phi((R_x\varphi) f, (R_x{\psi})e_1 h)\|dx \ \ \ (by  \ \eqref{prop})\\
&\leq&\int_{\T}\alpha \|\phi(R_x\varphi) f\| \|(R_x{\psi})e_1 )\|dx\\
&\leq& \alpha \|\varphi\|\|\psi\|.
\end{eqnarray*}
Hence by Lemma \ref{lem T hyp const}, we should have
$$|(\Phi\circ N_{f,h})(e_1-1)|\leq  12\sqrt{\pi(1+\sqrt{2})}\sqrt{\alpha},$$
which by \eqref{Nf}, it implies that
\begin{equation}\label{e_1}
\begin{aligned}
|\varphi(fe_1,h)-\varphi(f,e_1h)| &=&|\Phi(fe_1\otimes h-f\otimes e_1 h)|\\
&\leq& 12\sqrt{\pi(1+\sqrt{2})}\sqrt{\alpha}.
\end{aligned}
\end{equation}
Now we show that
\begin{align}\label{e_n}
|\phi(fe_n,h)-\phi(f,e_nh)|\leq12\sqrt{\pi(1+\sqrt{2})}\sqrt{\alpha} \|f\| \|h\|
\end{align}
for all $f,h\in A(\T)$, where $e_n$ denotes the function in $A(\T)$ defined by
$$e_n(s)=e^{ins} \ \ (s\in \R, \ n\in \mathbb{Z}).$$
For $a\in A(\T)$, let $a_n\in A(\T)$ be the function defined by
$$a_n(x)=a(nx).$$
Note that $e_n=(e_1)_n.$
Define $\tau: A(\T)\times A(\T)\rightarrow \mathbb{C}$ by
$$\tau(a,b)=\phi(fa_n, hb_n) \ \ (a,b\in A(\T)).$$
Note that if $a\in A(\T),$ then $a(s)=\sum_{k=-\infty}^{+\infty}\hat{a}(k)e^{iks},$ hence $a(ns)=\sum_{k=-\infty}^{+\infty}\hat{a}(k)e^{ikns}$ and so $a_n\in A(\T)$ with
$$\|a_n\|\leq\sum_{k=-\infty}^{+\infty}|\hat{a}(k)|= \|a\|.$$
Moreover, if $a,b \in A(\T)$ are such that $\text{supp}\,a\cap \text{supp}\,b=\emptyset$, then it is easily seen that $\text{supp}\,fa_n\cap \text{supp}\,hb_n=\emptyset$. So
\begin{eqnarray*}
|\tau(a,b)|&\leq& \|\phi(fa_n,hb_n)\|\\
&\leq& \alpha \|fa_n\| \|hb_n\|\\
&\leq& \alpha \|a\| \|b\|.
\end{eqnarray*}
From \eqref{e_1}, we deduce that
\begin{align}\label{e1 to en}
|\tau(e_1,1)-\tau(1,e_1)|\leq 12\sqrt{\pi(1+\sqrt{2})}\sqrt{\alpha}.
\end{align}
On the other hand, we have
$$\tau(e_1,1)=\phi(fe_n,h), \ \tau(1,e_1)=\phi(f,e_n h)$$
which, together with \eqref{e1 to en}, gives \eqref{e_n}.\\
Now let $g\in A(\T).$ Since $g=\sum_{k=-\infty}^{+\infty}\hat{g}(k)e_k,$ by applying \eqref{e_n} we get
\begin{eqnarray*}
|\phi(fg,h)-\phi(f,gh)|&=&|\phi(\sum_{k=-\infty}^{+\infty}\hat{g}(k)f e_k,h)-\phi(f,\sum_{k=-\infty}^{+\infty}\hat{g}(k)e_kh)|\\
&\leq&\sum_{-\infty}^{+\infty}|\hat{g}(k)| |\phi(fe_k,h)-\phi(f,e_k h)|\\
&\leq& \sum_{-\infty}^{+\infty}|\hat{g}(k)| 12\sqrt{\pi(1+\sqrt{2})}\sqrt{\alpha}\\
&=&12\sqrt{\pi(1+\sqrt{2})}\sqrt{\alpha} \|g\|.
\end{eqnarray*}
Therefore if $f,h\in A(\T)$ are arbitrary elements, we get
\begin{align}\label{final ineq sqr}
|\phi(fg,h)-\phi(f,gh)|\leq 12\sqrt{\pi(1+\sqrt{2})}\sqrt{\alpha} \|f\| \|g\| \|h\|.
\end{align}
Next, let $m: A(\T\times\T)\rightarrow A(\T)$ be the multiplication map which maps every elementary tensor $f\otimes g\in A(\T\times \T)$ to $fg\in A(\T)$. It follows from \eqref{final ineq sqr} that for $u=\sum_{i=1}^{\infty}f_i\otimes g_i\in A(\T\times\T)$ we can write
\begin{eqnarray*}
|\Phi(u)-\phi(1,m(u))|&=& |\Phi(\sum_{i=1}^{\infty}f_i\otimes g_i-\sum_{i=1}^{\infty}1\otimes f_i g_i)|\\
&\leq& 12\sqrt{\pi(1+\sqrt{2})}\sqrt{\alpha} \sum_{i=1}^{\infty} \|f_i\|\|g_i\|,
\end{eqnarray*}
In particular, for every $u\in I:=\ker\, m,$
$$|\Phi(u)|\leq 12\sqrt{\pi(1+\sqrt{2})}\sqrt{\alpha} \|u\|,$$
implying that
\begin{align}\label{*}
\|\Phi|_I\|\leq 12\sqrt{\pi(1+\sqrt{2})}\sqrt{\alpha}.
\end{align}
Now consider the general case. Let $\phi: A(\T)\times A(\T)\rightarrow \mathbb{C}$ be a continuous bilinear map satisfying \eqref{prop} for some $\alpha>0.$
Without lost of generality, we can assume that $\Phi|_I\neq0.$ Let $\Phi_0\in I^*$ with $\Phi_0=\frac{\Phi|_I}{\|\Phi|_I\|}.$ Then $\|\Phi_0\|=1$. By the Hahn-Banach Theorem, $\Phi_0$ can be extended to $\Psi\in A(\T\times\T)^*$ with $\|\Psi\|=1.$ For $f,g \in A(\T)$ with $\text{supp}\,f\cap \text{supp}\,g=\emptyset$ we have
\begin{eqnarray*}
|\Psi(f\otimes g)|&=&|\Phi_0(f\otimes g)|\\
&=& \frac{1}{\|\Phi|_I\|} |\Phi(f\otimes g)|\\
&\leq& \frac{\alpha}{\|\Phi|_I\|} \|f\| \|g\|.
\end{eqnarray*}
Put $\alpha_0=\frac{\alpha}{\|\Phi|_I\|}.$ Then $\|\Psi\|=1$ and $0\leq\alpha_0\leq 1$ (We can assume $\alpha\leq \|\Phi|_I\|$, otherwise the statement is trivial). By the first part and \eqref{*},
\begin{eqnarray*}
1=\|\Phi_0\|&=&\|\Psi|_I\|\\
&\leq& 12\sqrt{\pi(1+\sqrt{2})}\sqrt{\alpha_0}\\
&=& 12\sqrt{\pi(1+\sqrt{2})}\sqrt{\frac{\alpha}{\|\Phi|_I\|}}.
\end{eqnarray*}
This implies that
$$\|\Phi|_I\|\leq 144\pi(1+\sqrt{2})\alpha.$$
In particular, for every $u\in I$
$$|\Phi(u)|\leq 144\pi(1+\sqrt{2})\alpha \|u\|.$$
Finally for $f,g,h\in A(\T)$, it is clear that $fg\otimes h-f\otimes gh\in I.$ So we can write
\begin{eqnarray*}
|\phi(fg,h)-\phi(f,gh)|&=&|\Phi(fg\otimes h-f\otimes gh)|\\
&\leq& 144\pi(1+\sqrt{2})\alpha \|fg\otimes h-f\otimes gh\|\\
&\leq& 288\pi(1+\sqrt{2})\alpha \|f\|\|g\|\|h\|.
\end{eqnarray*}
\end{proof}

One important application of Theorem \ref{mainn thm for strong (B) constant of A(T)} is
to obtain a constant for the strong property $(\B)$ of C$^*$-algebras and group algebras. The approach we need to use is the same as that provided in \cite{AEV} with a slight modification using our result from the preceding subsection. Hence we do not give a proof to the following theorem and we just refer to \cite[Theorem 3.4]{AEV} and \cite[Theorem 3.5]{AEV} (see also \cite[Lemma 3.2]{AEV}). We highlight that the approach in \cite{AEV} does not give a constant for the strong property $(\B)$, whereas our modification does as the following theorem shows.\\
\begin{thm}\label{C* alg and group alg constant B}
Let $A$ be a C$^*$-algebra or a group algebra. Then $A$ has the strong property $(\B)$ with a constant given by
$$288\pi(1+\sqrt{2}).$$
\end{thm}

\section{A bound for the hyperreflexivity constant}\label{hyper constant c-star and grp}
In this section, we show how the existence of a constant for the strong property $(\B)$ can help us to find an upper bound for the hyperreflexivity constant of the bounded $n$-cocycle spaces. We then apply our result to  C$^*$-algebras and group algebras. We achieve our goal by modifying the approach in \cite{ES-JS} and its main result. We start by stating without proof the following proposition which is a  straightforward modification of \cite[Proposition 3.5]{ES-JS}.

\begin{prop}\label{P:Prop B strong-case n=12}
Let $A$ be a unital Banach algebra having the strong property
$(\B)$ with a constant $r.$\\
$($i$)$ For every right Banach $A$-module $X$ and a bounded operator
$D \from A \to X$ and each $\alpha\geq0$ satisfying
$$ab=0 \Rightarrow \|D(a)b\|\leq \alpha \|b\| \|a\|$$
we have
$$  \|D(ab)c-D(a)bc\|\leq r\alpha \|a\| \|b\| \|c\| \hspace{4mm} (\forall a,b,c \in A).$$
$($ii$)$ For every right Banach $A$-module $X$ and a bounded operator
$D \from A \to X$ and each $\beta\geq 0$ satisfying
$$ab=bc=0 \Rightarrow \|aD(b)c\|\leq \beta \|a\|\|b\|\|c\|$$
we have
$$\|d[D(acb)-aD(cb)-D(ac)b+aD(c)b]e\|\leq r^2\beta \|a\|\|b\|\|c\|\|d\|\|e\| \hspace{4mm} (\forall a,b,c,d,e \in A).$$
\end{prop}

Our next step is to improve Proposition \ref{P:Prop B strong-case n=12} $(ii)$ to higher dimensions using the induction as it is demonstrated in the following Theorem. This is a modification of \cite[Theorem 3.6]{ES-JS}, taking into account our concept of the strong property $(\B)$ with a constant. We present the proof for the sake of clarity.

\begin{thm}\label{T:Loc. distance n cocycle2}
Let $A$ be a unital Banach algebra with unit 1 having the strong property
$(\B)$ with a constant $r.$
Suppose that $X$ is a unital Banach $A$-bimodule, $n\in \mathbb{N} $, $T\in  B^n(A,X)$ and let $\gamma\geq 0$ satisfying
$$a_0a_1=a_1a_2=\cdots=a_na_{n+1}=0\Rightarrow \|a_0T(a_1,\ldots,a_n)a_{n+1}\| \leq \gamma \|a_0\|\cdots\|a_{n+1}\|.$$
Also $T(a_1,\ldots,a_n)=0$ if for some $1\leq i \leq n, a_i=1$.
Then
$$\|\delta^n(T)\|\leq 2^{n-1}r^{n+1}\gamma.$$
\end{thm}

\begin{proof}
We prove the statement by induction on $n$. For $n=1$, the result follows from
Proposition \ref{P:Prop B strong-case n=12}(ii) together with the fact that $X$ is unital
and $T(1)=0$.

Now suppose that the result is true for $n\in \N$. We prove it for $n+1$.
Consider $T\in B^{n+1}(A,X)$ and $ \gamma\geq0$ satisfying
$$a_0a_1=a_1a_2=\cdots=a_{n+1}a_{n+2}=0\Rightarrow||a_0T(a_1,\ldots, a_{n+1})a_{n+2}\|\leq \gamma \|a_0\|\cdots\|a_{n+2}\|.$$
Also $T(a_1,\ldots,a_{n+1})=0$ if for some $1\leq i \leq n+1, a_i=1$. Take $a_i\in A$, $i=0,\ldots, n+1$
with $a_0a_1=a_1a_2=\cdots=a_{n}a_{n+1}=0$.
We first show that
\begin{align}\label{Eq:Prop B strong-1}
\|a_0\star\Lambda_n(T)(a_1,\ldots,a_{n})\star a_{n+1}\|\leq r\gamma \|a_0\|\cdots\|a_{n+1}\|.
\end{align}
First suppose that $\|a_0\|=\cdots=\|a_{n+1}\|=1$, and let $$S=a_0\star\Lambda_n(T)(a_1,\ldots,a_n)\star a_{n+1}.$$
For every $b,c\in A$ with $bc=0$, we have
\begin{eqnarray*}
S(b)c&=&[a_0\star\Lambda_n(T)(a_1,\ldots,a_n)\star a_{n+1}](b)c\\
&=&a_0 \Lambda_n(T)(a_1,\ldots,a_n)(a_{n+1}b)c-a_0 \Lambda_n(T)(a_1,\ldots,a_n)(a_{n+1})bc\\
&=&a_0T(a_1,\ldots,a_n,a_{n+1}b)c.
\end{eqnarray*}
But $a_0a_1=\cdots=a_n(a_{n+1}b)=(a_{n+1}b)c=0$. Thus, by our hypothesis
$$\|a_0T(a_1,\ldots,a_n,a_{n+1}b)c\|\leq \gamma \|a_0\|\cdots\|a_{n+1}b\|\|c\|\leq \gamma\|b\|\|c\|$$
implying that $\|S(b)c\|\leq \gamma  \|b\|\|c\|$. Hence, by Proposition \ref{P:Prop B strong-case n=12}(i), we get
\begin{align}\label{Eq:Prop B strong-2}
\|S(bc)-S(b)c\|\leq r\gamma \|b\| \|c\| \hspace{4mm} (\forall b,c\in A).
\end{align}
On the other hand,
\begin{eqnarray*}
S(1)&=&(a_0\star\Lambda_n(T)(a_1,\ldots,a_n)\star a_{n+1})(1)\\
&=&a_0\Lambda_n(T)(a_1,\ldots,a_n)(a_{n+1}1)-a_0\Lambda_n(T)(a_1,\ldots,a_n)(a_{n+1})1\\
&=&0.
\end{eqnarray*}
Putting $b=1$ in \eqref{Eq:Prop B strong-2}, we can write
$$\|S(c)\|\leq r\gamma \|c\| \hspace{4mm} ( c\in A)$$
or equivalently,
\begin{align}\label{Eq:Prop B strong-3}
\|S\|=\|a_0\star\Lambda_n(T)(a_1,\ldots,a_n)\star a_{n+1}\|\leq r\gamma.
\end{align}
Now consider the general case. If for some $0\leq i\leq n+1$, $a_i=0$, then we clearly have
$$\|a_0\star\Lambda_n(T)(a_1,\ldots,a_n)\star a_{n+1}\|\leq r\gamma \|a_0\|\ldots\|a_{n+1}\|.$$
Now suppose that for all $0\leq i\leq n+1$, $a_i\neq 0$. Then
$$\frac{a_0}{\|a_0\|}\frac{a_1}{\|a_1\|}=\cdots=\frac{a_{n+1}}{\|a_{n+1}\|}\frac{a_{n+2}}{\|a_{n+2}\|}=0,$$
and so, by \eqref{Eq:Prop B strong-3},
$$\|\frac{a_0}{\|a_0\|}\star\Lambda_n(T)(\frac{a_1}{\|a_1\|},\ldots,\frac{a_n}{\|a_n\|})\star \frac{a_{n+1}}{\|a_{n+1}\|}\|\leq r\gamma ,$$
implying that \eqref{Eq:Prop B strong-1} holds.

Now let $B_A(A,X)$ denote the space of all left multipliers from $A$ into $X$ and suppose that
$q: B(A,X)\to \displaystyle \frac{B(A,X)}{B_A(A,X)}$ is the natural quotient mapping.
It is straightforward to verify that $\displaystyle \frac{B(A,X)}{B_A(A,X)}$ is a unital Banach $A$-bimodule and
$q$ is an $A$-bimodule morphism. Thus, by \eqref{Eq:Prop B strong-1},
\begin{eqnarray*}
\|a_0\star q(\Lambda_n(T)(a_1,\ldots,a_n))\star a_{n+1}\|&=&\| q(a_0\star\Lambda_n(T)(a_1,\ldots,a_n)\star a_{n+1})\| \\
&\leq&\|q\|\| a_0\star\Lambda_n(T)(a_1,\ldots,a_n)\star a_{n+1}\|\\
&\leq& r\gamma \|a_0\|\ldots\|a_{n+1}\|.
\end{eqnarray*}
Moreover, if for some $i,1\leq i\leq n$, $a_i=1$, then for every $a\in A$,
$$ \Lambda_n(T)(a_1,\ldots,a_n)(a)=T(a_1,\ldots,a_n,a)=0.$$
This shows that $q\circ \Lambda_n(T)(a_1,\ldots,a_n)=0$ if for some $1\leq i\leq n$, $a_i=1$.
Now using the assumption of the induction, we have
\begin{equation}\label{Eq:Prop B strong-4}
\begin{aligned}
\|\Delta_q^n(q\circ \Lambda_n(T))(a_1,\ldots,a_{n+1})\|&\leq& (2^{n-1}r^{n+1})(r\gamma)\|a_1\|\cdots\|a_{n+1}\|\\
&\leq& 2^{n-1}r^{n+2}\gamma \|a_1\|\cdots\|a_{n+1}\|
\end{aligned}
\end{equation}
where $\Delta_q^n: \displaystyle B^{n} (A,\frac{B(A,X)}{B_A(A,X)})\rightarrow  B^{n+1}(A,\frac{B(A,X)}{B_A(A,X)})$ is the corresponding connecting map in Definition \ref{connecting maps}.
On the other hand, since $q$ is a Banach $A$-bimodule morphism, it is easy to check that for all $a_0,\ldots,a_{n+1}\in A$,
\begin{eqnarray*}
\Delta_q^n(q\circ \Lambda_n(T))(a_1,\ldots,a_{n+1}) &=&  q(\Delta^n(\Lambda_n(T))(a_1,\ldots,a_{n+1})) \\
&=& q(\Lambda_{n+1}(\delta^{n+1}(T))(a_1,\ldots,a_{n+1})).
\end{eqnarray*}
Hence, by \eqref{Eq:Prop B strong-4}
\begin{align*}
\|q(\Lambda_{n+1}(\delta^{n+1}(T))(a_1,\ldots,a_{n+1}))\|\leq 2^{n-1}r^{n+2}\gamma\|a_1\|\cdots\|a_{n+1}\|
\end{align*}
implying that for $S=\Lambda_{n+1}(\delta^{n+1}(T))(a_1,\ldots,a_{n+1})$,
$$\|\dist(S, B_A(A,X))\| \leq 2^{n-1}r^{n+2}\gamma\|a_1\|\cdots\|a_{n+1}\|. $$
So for every $a\in A$, we have
\begin{equation}\label{Eq:Prop B strong-5}
\begin{aligned}
\|S(a)-S(1)a\|&\leq& 2[2^{n-1}r^{n+2}\gamma\|a_1\|\cdots\|a_{n+1}\|\|a\|]\\
&\leq& 2^{n}r^{n+2}\gamma\|a_1\|\cdots\|a_{n+1}\|\|a\|.
\end{aligned}
\end{equation}
On the other hand,
\begin{eqnarray*}
S(1)&=&\Lambda_{n+1}(\delta^{n+1}(T))(a_1,\ldots,a_{n+1})(1)\\
&=&\delta^{n+1}(T)(a_1,\ldots,a_{n+1},1)\\
&=&a_1T(a_2,\ldots,a_{n+1},1)+\sum_{j=0}^{n-1}(-1)^jT(a_1,\ldots,a_j a_{j+1},\ldots,a_{n+1},1)+(-1)^{n}T(a_1,\ldots,a_{n+1}1)\\
&+&(-1)^{n+1}T(a_1,\ldots,a_{n+1})1\\
&=& 0.
\end{eqnarray*}
Therefore by putting $a=a_{n+2}$ in \eqref{Eq:Prop B strong-5}, we have
\begin{eqnarray*}
\|\delta^{n+1}(T)(a_1,\ldots,a_{n+2})\|&=&\|\Lambda_{n+1}(\delta^{n+1}(T))(a_1,\ldots,a_{n+1})(a_{n+2}) \\
&=& \|S(a_{n+2})\| \\
&=&\|S(a_{n+2})-S(1)a_{n+2}\|\\
&\leq& 2^{n}r^{n+2}\gamma\|a_1\|\ldots\|a_{n+2}\|.
\end{eqnarray*}
This completes the proof.

\end{proof}

Now we are ready to give the main result of this section.

\begin{thm}\label{T:unit prop B strong-vanish coho-n hyperref2}
Let $A$ be a Banach algebra having b.l.u. and the strong property $(\B)$ with a constant $r$. Let $M$ be a bound for the local units of $A$. Let $n\in \N$, suppose that $X$ is a Banach $A$-bimodule such that
$\Hc^{n+1}(A,X)$ is a Banach space. Then for each $T\in B^n(A,X),$ we have
 $$\dist(T,\Zc^n(A,X))\leq C2^{n-1}(M^2r+(M+1)^2)^{n+1}dist_r(T,\Zc^n(A,X))$$
where $C$ is a constant satisfying
\begin{align}\label{inequality to constant by open mapping thm}
\dist(T, Z^n(A,X))\leq C\|\delta^n(T)\|, \ \ \ (T\in B^n(Z,X)).
\end{align}
 \end{thm}
\begin{proof}
Let $T\in B^n(A,X)$. By \cite[Lemma 3.7]{ES-JS},
for every $a_i \in A^\sharp$, $i=0,\ldots,n+1$ with $a_0a_1=\cdots=a_na_{n+1}=0$, we have
$$\|a_0\sigma(T)(a_1,\ldots,a_n)a_{n+1}\| \leq dist_r(T,\Zc^n(A,X)) \|a_1\|\cdots \|a_{n+1}\|$$
where $\sigma(T):A^{\sharp}\rightarrow X$ is defined by
$$\sigma(T)(b_1+\lambda_1,\ldots,b_n+\lambda_n)=T(b_1,\ldots,b_n) \ \ \ (b_i\in A,\ \lambda_i\in \mathbb{C}).$$
On the other hand, if for some $1\leq i\leq n, \ \ a_i=1$, then
$$\sigma(T)(a_1,\ldots,a_n)=0.$$
If we apply \cite[Theorem 5.3]{ES-JS}, we find that $A^{\sharp}$ has the strong property $(\B)$ with a constant given by,
$$M^2r+(M+1)^2.$$
Hence we can use Theorem \ref{T:Loc. distance n cocycle2} to write
\begin{align}\label{delta-distr}
\|\delta^{\sharp n}(\sigma(T))\|\leq 2^{n-1}(M^2r+(M+1)^2)^{n+1}\dist_r(T,\Zc^n(A,X)).
\end{align}

Now since $\mathcal{H}^{n+1} (A,X)$ is a Banach space, $\text{Im} \delta^n$ is closed. Hence, by the open mapping theorem, there is a constant
$C>0$ such that for each $T\in B^n(A,X),$
\begin{align}\label{dist-delta n}
\dist(T, \Zc^n(A,X))\leq C \|\delta^n (T)\|.
\end{align}
It is straightforward to check that
\begin{align}\label{delta-delta sharp}
\|\delta^n (T)\|\leq \|\delta^{\sharp n} (\sigma(T))\|.
\end{align}
 Hence putting \eqref{delta-distr}, \eqref{dist-delta n} and \eqref{delta-delta sharp} together we get
 \begin{align}\label{dist to distr}
\dist(T, \Zc^n(A,X))\leq C 2^{n-1}(M^2r+(M+1)^2)^{n+1}\dist_r(T,\Zc^n(A,X)),
\end{align}
as desired.
\end{proof}

\subsection{C$^*$-algebras and group algebras}

We showed in Section \ref{constant for the strong property (B)circle} that every C$^*$-algebra and group algebra has the strong property $(\B)$ with a constant. On account of Theorem \ref{T:unit prop B strong-vanish coho-n hyperref2}, this enables us to obtain an upper bound for the hyperreflexivity constant of the bounded $n$-cocycle spaces of C$^*$-algebras and group algebras. First we need to introduce the notion of amenability constant.
\begin{defn}
Let $A$ be a Banach algebra. The amenability constant of $A$, which we denote
by $AM(A)$, is
$$\inf\{\sup_{\alpha}\|\mu_{\alpha}\| : (\mu_{\alpha})_{\alpha} \ \text{ is a bounded approximate diagonal for $A$}\}$$
where we define the infimum of the empty set to be $+\infty$. Hence $A$ is amenable if and only if $AM(A)<\infty$.
\end{defn}
\begin{rem}\label{amen const for group alg and C*}
$(i)$ Let $G$ be a locally compact amenable group. Then $AM(L^1(G))=1$ (see \cite[Corollary 1. 10]{Stokke}).\\
$(ii)$ Let $A$ be an amenable C$^*$-algebra. Then $AM(A)=1.$
\end{rem}
\begin{rem}\label{C and amen const}
Let $A$ be an amenable Banach algebra and suppose that $X$ is a dual Banach $A$-bimodule. Then $C\leq AM(A)$ where $C$ is the constant given in \eqref{inequality to constant by open mapping thm} (see \cite{AEV2}).
\end{rem}
\begin{thm}\label{hyper constant for C* and group alg}
Suppose that $A$ is a C$^*$-algebra or the group algebra of a group with an open subgroup of polynomial growth. Let $n\in \N$, and let $X$ be a Banach $A$-bimodule such that
$\Hc^{n+1}(A,X)$ is a Banach space. Then $\Zc^n(A,X)$ is hyperreflexive with a constant bounded by
$$C 2^{n-1}(288\pi M^2(1+\sqrt{2})+(M+1)^2)^{n+1}$$
where $M$ is a bound for the local units of $A$ and $C$ is a constant satisfying in \eqref{inequality to constant by open mapping thm}. In particular, we have\\
$(i)$ If $A$ is a C$^*$-algebra or the group algebra of a discrete group, then $M=1$.\\
$(ii)$ In the case where $A$ is amenable and $X$ is the dual of a Banach $A$-bimodule, we can assume that $C=1$.
\end{thm}

\begin{proof}
The main statement follows if we combine Theorem \ref{C* alg and group alg constant B} and Theorem \ref{T:unit prop B strong-vanish coho-n hyperref2}. To prove $(i)$, note that the group algebra of a discrete group is unital. Moreover, Proposition \cite[Proposition 6.1]{ES-JS} shows that local units of a C$^*$-algebra are bounded by 1. Finally, $(ii)$ follows if we apply Remark \ref{amen const for group alg and C*} and Remark \ref{C and amen const}.
\end{proof}

\subsection{Convolution operators}
For a locally compact group $G$ and $1<p<\infty$, we recall that an operator $T\in B(L^p(G))$ is
a {\it convolution operator} if for every $t\in G$ and $f\in L^p(G)$, $T(\delta_t*f)=\delta_t*T(f)$.
The space of all convolution operators on $L^p(G)$ is denoted by $CV_p(G)$. It is straightforward to check
that $CV_p(G)$ is a $w^*$-closed subalgebra of $B(L^p(G))\cong (L^p(G)\widehat{\otimes} L^q(G))^*$, where $q$ is the conjugate of $p$.

In this section, we discuss reflexivity and hyperreflexivity of $CV_p(G)$. For the latter case, our method once again relies on the concept of the strong property $(\B)$ with a constant applied to group algebras.
But first, we need the following theorem which is the generalization of a classical result of
E. Christensen in \cite{Christen} for $C^*$-algebras. We recall that for a Banach algebra $A$, a Banach
$A$-bimodule $X$, and $x\in X$, the inner derivation $\delta_x:A \to X$ is definded by
$\delta_x(a)=a\cdot x-x\cdot a$.

\begin{thm}\label{rep-hyper}
Let $A$ be a Banach algebra with an approximate identity bounded by $K$ and having the strong property $(\B)$ with constant $C$, $X$ a Banach space, and $\pi:A\rightarrow B(X) $ a continuous non-degenerate representation. If $\Hc^1(A,B(X))$ is a Banach space, then $\pi(A)'$, the commutant of $\pi(A)$, is hyperreflexive and its hyperreflexivity constant is bounded by $MCK^2\|\pi\|^2$, where $M$ is a constant satisfying in
\begin{align}\label{open map const}
dist(L, \pi(A)')\leq M\|\delta_L\|, \ \ \ L\in B(X).
\end{align}
\end{thm}

\begin{proof}
Let $T\in B(X)$ and $\alpha=dist_r(T,\pi(A)').$ Fix $x\in X$ and define
$$\varphi_{T,x}: A\times A\rightarrow X, \ \ \ (a,b)\mapsto (\pi(a)\circ T\circ \pi(b))(x).$$
If $a,b\in A$ with $ab=0$, then for $S\in \pi(A)',$
\begin{eqnarray*}
\|\varphi_{T,x}(a,b)\|&=&\|\pi(a)T(\pi(b)(x))-\pi(a)S(\pi(b)(x))\|\\
&=&\|\pi(a)(T-S)(\pi(b)(x))\|\\
&\leq&\|\pi(a)\|\|(T-S)(\pi(b)(x))\|.
\end{eqnarray*}
Therefore
\begin{eqnarray*}
\|\varphi_{T,x}(a,b)\|&\leq& \|\pi(a)\| \inf_{S\in A'}\|(T-S)(\pi(b)(x))\|\\
&\leq& \|\pi(a)\|\alpha \|\pi(b)(x)\|\\
&\leq&\alpha\|\pi\|^2\|x\|\|a\|\|b\|.
\end{eqnarray*}
Since $A$ has the strong property $(\B)$ with a constant $C$, we have that for every $a,b,c\in A$,
\begin{eqnarray}\label{Eq:hyper-rep}
\|\varphi_{T,x}(ab,c)-\varphi_{T,x}(a,bc)\|\leq C\alpha\|\pi\|^2\|x\| \|a\|\|b\|\|c\|.
\end{eqnarray}
However, if $\{e_i\}_{i\in I}$ is an approximate identity in $A$ bounded by $K$, then it follows from
the non-degeneracy of $\pi$ that
$\lim_{i\to \infty} \pi(e_i)=id_X$, where the convergence happens in the strong operator topology
of $B(X)$. Hence if we put $a=c=e_i$ in \eqref{Eq:hyper-rep} and let $i\to \infty$, we have
\begin{eqnarray*}
\| \delta_T(b)(x) \| &=& \| (T\cdot b-b\cdot T)(x)\| \\
&=& \| T(\pi(b)(x))-\pi(b)T(x)\| \\
&=& \lim_{i\to \infty} \| \pi(e_ib)T(\pi(e_i)(x))-\pi(e_i)T(\pi(be_i)(x)) \| \\
&=& \lim_{i\to \infty} \| \varphi_{T,x}(e_ib,e_i)-\varphi_{T,x}(e_i,be_i) \| \\
&\leq & C K^2 \alpha\|\pi\|^2\|x\|\|b\|.
\end{eqnarray*}
Since $x\in X$ was arbitrary, we can conclude that for every $b\in A$,
\begin{align}\label{ref-ineq}
\|\delta_T(b)\|\leq C\alpha\|\pi\|^2 K^2 \|b\|.
\end{align}
Now we define
$$\phi: B(X)\rightarrow \Zc^1(A,B(X)), \ \ \ L\mapsto\delta_L.$$
Using \eqref{ref-ineq} we get
\begin{align}\label{second}
\|\delta_T\|\leq C\alpha\|\pi\|^2.
\end{align}
But $\Hc^1(A,B(X))=\frac{\Zc^1(A,B(X))}{\mathcal{N}^1(A,B(X))}$ is a Banach space so that $\mathcal{N}^1(A,B(X))=Im\phi$ is closed. Hence applying the open mapping theorem to the map $i: \frac{B(X)}{ker\phi}\rightarrow \mathcal{N}^1(A,B(X))$, we find $M>0$ such that for all $L\in B(X),$
$$dist(L, ker\phi)\leq M\|\delta_L\|$$
which, together with \eqref{second}, gives
$$dist(T, ker\phi)\leq MCK^2\|\pi\|^2\alpha= MCK^2\|\pi\|^2\dist_r(T,\pi(A)').$$
The final result follows since $ker\phi=\pi(A)'.$
\end{proof}


It is well-known that all von Neumann algebras are reflexive and injective von Neumann algebras are hyperreflexive. The former is direct and simple application of the double commutate theorem whereas as the latter is due to the beautiful work of E. Christensen \cite{Christen}. In particular, for a locally compact group $G$, $CV_2(G)$, its group von Neumann algebra, is reflexive and, when $G$ is amenable , it is hyperreflexive. In the following theorem, we extend this to all convolution operators on $L^p(G)$ for every $1<p<\infty$.

\begin{thm}
Let $G$ be a locally compact group, and let $1<p<\infty$. Then $CV_p(G)\subseteq B(L^p(G))$ is reflexive. If, in addition, $G$ is amenable, the $CV_p(G)$ is hyperreflexive and its hyperreflexivity constant is bounded by $288\pi(1+\sqrt{2}).$
\end{thm}

\begin{proof}
In Theorem \ref{rep-hyper}, let $A=L^1(G)$, $X=L^p(G)$ and $\pi=\rho$, the right regular representation
of $G$ on $L^p(G)$. Then all the assumptions of Theorem \ref{rep-hyper} are satisfied. Therefore $VCV_p(G)=\pi(l^1(G))'$ is hyperreflexive and its hyperreflexivity constant is bounded by $288\pi(1+\sqrt{2}).$
\end{proof}

\end{document}